\documentclass[reqno,11pt]{article} 
\usepackage[left=2.5cm,right=2.5cm,top=3cm,bottom=3cm,a4paper]{geometry}
\usepackage{amsmath, amssymb}
\usepackage{amsthm, amscd} 
\usepackage[all,cmtip]{xy}
\usepackage{float}
\usepackage{caption}
\usepackage{color}

\makeatletter
\DeclareRobustCommand{\rvdots}{%
  \vbox{
    \baselineskip4\p@\lineskiplimit\z@
    \kern-\p@
    \hbox{.}\hbox{.}\hbox{.}
  }}
\makeatother

\newcommand{\Mod}[1]{\ (\textup{mod}\ #1)}
\makeatletter
\def\moverlay{\mathpalette\mov@rlay}
\def\mov@rlay#1#2{\leavevmode\vtop{%
   \baselineskip\z@skip \lineskiplimit-\maxdimen
   \ialign{\hfil$\m@th#1##$\hfil\cr#2\crcr}}}
\newcommand{\charfusion}[3][\mathord]{
    #1{\ifx#1\mathop\vphantom{#2}\fi
        \mathpalette\mov@rlay{#2\cr#3}
      }
    \ifx#1\mathop\expandafter\displaylimits\fi}
\makeatother

\theoremstyle{plain} 
\newtheorem{theorem}{\indent\sc Theorem}[section]
\newtheorem{lemma}[theorem]{\indent\sc Lemma}

\newtheorem{proposition}[theorem]{\indent\sc Proposition}

\theoremstyle{definition} 

\newtheorem{remark}[theorem]{\indent\sc Remark}
\newtheorem{example}[theorem]{\indent\sc Example}

\title{Class fields generated by coordinates
of elliptic curves}
\author{
\textsc{Ho Yun Jung$^{*}$,  Ja Kyung Koo  and Dong Hwa Shin}
}
\date{} 
\begin{document}

\allowdisplaybreaks

\maketitle


\footnote{
\textbf{${}^*$Corresponding author: 
Ho Yun Jung:} Department of Mathematics, Dankook University, Cheonan-si, Chungnam 31116, Republic of Korea, E-mail: hoyunjung@dankook.ac.kr\\
\textbf{Ja Kyung Koo}, 
Department of Mathematical Sciences,
KAIST, Daejeon 34141, Republic of Korea,
E-mail: jkgoo@kaist.ac.kr\\
\textbf{Dong Hwa Shin:} Department of Mathematics, Hankuk University of Foreign Studies, Yongin-si, Gyeonggi-do 17035, Republic of Korea, E-mail: dhshin@hufs.ac.kr}


\begin{abstract}
Let $K$ be an imaginary quadratic field different from
$\mathbb{Q}(\sqrt{-1})$ and $\mathbb{Q}(\sqrt{-3})$.
For a nontrivial integral ideal $\mathfrak{m}$ of $K$,
let $K_\mathfrak{m}$ be the ray class field
modulo $\mathfrak{m}$. By using some
inequalities on special values of modular functions, we show that
a single $x$-coordinate of a certain elliptic curve
generates $K_\mathfrak{m}$ over $K$.
\end{abstract}


\maketitle

\section {Introduction}

Let $K$ be an imaginary quadratic field with ring of integers $\mathcal{O}_K$. And, 
let $E$ be the elliptic curve with complex multiplication by $\mathcal{O}_K$
given by the Weierstrass equation
\begin{equation*}
E~:~y^2=4x^3-{{g^{}_{2}}}x-{g^{}_{3}}\quad
\textrm{with}~{g^{}_{2}}={g^{}_{2}}(\mathcal{O}_K)~
\textrm{and}~{g^{}_{3}}={g^{}_{3}}(\mathcal{O}_K).
\end{equation*}
For $z\in\mathbb{C}$, let $[z]$ denote the
coset $z+\mathcal{O}_K$ in $\mathbb{C}/\mathcal{O}_K$.
Then the map
\begin{eqnarray*}
\varphi^{}_{K}~:~\mathbb{C}/\mathcal{O}_K&\rightarrow&E(\mathbb{C})\quad(\subset\mathbb{P}^2(\mathbb{C}))\\
\mathrm{[}z]~~ & \mapsto & [\wp(z;\,\mathcal{O}_K):\wp'(z;\,\mathcal{O}_K):1],
\end{eqnarray*}
where $\wp$ is the Weierstrass $\wp$-function relative to $\mathcal{O}_K$,
is a complex analytic isomorphism of
complex Lie groups (\cite[Proposition 3.6 in Chapter VI]{Silverman}).
Corresponding to $E$, we consider the Weber function
$\mathfrak{h}^{}_{K}:E(\mathbb{C})\rightarrow\mathbb{P}^1(\mathbb{C})$
given by
\begin{equation*}
\mathfrak{h}^{}_{K}(x,\,y)=
\left\{\begin{array}{ll}
\displaystyle\frac{{g^{}_{2}}{g^{}_{3}}}{\Delta}\,x & \textrm{if}~j(E)\neq0,\,1728,\vspace{0.1cm}\\
\displaystyle\frac{g_{2}^2}{\Delta}\,x^2 & \textrm{if}~j(E)=1728,\vspace{0.1cm}\\
\displaystyle\frac{{g^{}_{3}}}{\Delta}\,x^3 & \textrm{if}~j(E)=0,
\end{array}\right.
\end{equation*}
where $\Delta=g_{2}^3-27g_{3}^2$ ($\neq0$)
and $j(E)=j(\mathcal{O}_K)$ is the $j$-invariant of $E$.
For a nontrivial ideal $\mathfrak{m}$ of $\mathcal{O}_K$, by $K_\mathfrak{m}$ we mean
the ray class field of $K$ modulo $\mathfrak{m}$.
In particular, $K_{\mathcal{O}_K}$ is the Hilbert class field $H_K$ of $K$.
Then we get by the theory of complex multiplication that
$H_K=K(j(E))$ and
\begin{equation*}
K_\mathfrak{m}=H_K\left(\mathfrak{h}^{}_{K}(x,\,y)\right)~
\textrm{for some $\mathfrak{m}$-torsion point $(x,\,y)$ on $E$}
\end{equation*}
if $\mathfrak{m}$ is proper (\cite[Chapter 10]{Lang87}).
In a letter to Hecke concerning Kronecker's Jugendtraum (= Hilbert 12th problem), Hasse asked whether
every abelian extension of $K$ can be generated only by a single
value of the Weber function $\mathfrak{h}^{}_{K}$ over $K$
(\cite[p. 91]{F-L-R}) and
Sugawara first gave partial answers to this question (\cite{Sugawara33} and \cite{Sugawara36}).
Recently, Jung, Koo and Shin (\cite{J-K-S}) proved that
if $\mathfrak{m}=N\mathcal{O}_{K}$ and $N\in\{2,\,3,\,4,\,6\}$, then
\begin{equation}\label{N2N}
K_{\mathfrak{m}}=K\left(\mathfrak{h}^{}_{K}(\varphi^{}_{K}([\tfrac{1}{N}]))\right)\quad\textrm{or}\quad
K_{\mathfrak{m}}=K\left(\mathfrak{h}^{}_{K}(\varphi^{}_{K}([\tfrac{2}{N}]))\right).
\end{equation}
And, Koo, Shin and Yoon (\cite{K-S-Y}) further showed
by utilizing the second Kronecker's limit formula that
(\ref{N2N}) holds for $N=5$ and $N\geq7$.
Besides, it is worth noting that Ramachandra (\cite{Ramachandra})
constructed a complicated primitive generator of
$K_\mathfrak{m}$ over $K$ by using
special values of the product of high powers of the discriminant $\Delta$ function and Siegel functions, which is beautiful in theory.
\par
Now, we assume that $K$ is different from $\mathbb{Q}(\sqrt{-1})$
and $\mathbb{Q}(\sqrt{-3})$, and so
${g^{}_{2}}{g^{}_{3}}\neq0$ and $j(\mathcal{O}_K)\neq0,\,1728$
(\cite[p. 200]{Cox}). Let
$\{E_{K,\,n}\}_{n\in\mathbb{Z}_{\geq0}}$ be the family of elliptic curves isomorphic to $E$
given by the affine models
\begin{equation}\label{modelE}
E_{K,\,n}~:~y^2=4x^3-\frac{J_K(J_K-1)}{27}\,C_K^{2n}x-\frac{J_K(J_K-1)^2}{27^2}\,C_K^{3n}
\end{equation}
where
\begin{equation*}
J_K=\frac{1}{1728}\,j(\mathcal{O}_K)
\quad\textrm{and}\quad C_K=J_K^2(J_K-1)^3.
\end{equation*}
Then for each $n\in\mathbb{Z}_{\geq0}$ we have a parametrization
\begin{eqnarray*}
\mathbb{C}/\mathcal{O}_K&\rightarrow&E_{K,\,n}(\mathbb{C})
\quad(\subset\mathbb{P}^2(\mathbb{C}))\\
\mathrm{[}z]~~&\mapsto&\left[{x^{}_{K,\,n}}(z):{y^{}_{K,\,n}}(z):1\right]
\end{eqnarray*}
with
\begin{equation*}
{x^{}_{K,\,n}}(z)=C_K^n\frac{{g^{}_{2}}{g^{}_{3}}}{\Delta}\wp(z;\,\mathcal{O}_K)
\quad\textrm{and}\quad
{y^{}_{K,\,n}}(z)=\sqrt{\left(C_K^n\frac{{g^{}_{2}}{g^{}_{3}}}{\Delta}\right)^3}\wp'(z;\,\mathcal{O}_K).
\end{equation*}
Here we note that
\begin{equation}\label{xchp}
{x^{}_{K,\,n}}(z)=C_K^n\mathfrak{h}^{}_{K}\left(\varphi^{}_{K}\left([z]\right)\right)
\quad(z\in\mathbb{C}).
\end{equation}
Let $\mathfrak{m}$ be a proper nontrivial ideal of $\mathcal{O}_K$
in such a way that $K_\mathfrak{m}$ properly contains $H_K$.
And, let $\omega$ be an element of $K$
so that $[\omega]=\omega+\mathcal{O}_K$
is a generator of the $\mathcal{O}_K$-module $\mathfrak{m}^{-1}\mathcal{O}_K/\mathcal{O}_K$.
In this paper, we shall prove the following three assertions
(Theorems \ref{xy}, \ref{conditional} and \ref{inertcase})\,:

\begin{enumerate}
\item[(i)] We have
$K_\mathfrak{m}=K\left({x^{}_{K,\,n}}(\omega),\,{y^{}_{K,\,n}}(\omega)^2\right)$ for every $n\in\mathbb{Z}_{\geq0}$.
\item[(ii)] We get $K_\mathfrak{m}=K\left({x^{}_{K,\,n}}(\omega)\right)$
for every $n\in\mathbb{Z}_{\geq0}$ satisfying
\begin{equation*}
n\geq\frac{\frac{13}{24}\pi\sqrt{|d_K|}+
6\ln\left(\frac{229}{76}N_\mathfrak{m}\right)}{\frac{5}{2}\pi\sqrt{|d_K|}-\ln 877383}-\frac{1}{6}
\end{equation*}
where $d_K$ is the discriminant of $K$ and $N_\mathfrak{m}$ is the least positive integer in $\mathfrak{m}$.
\item[(iii)] If $\mathfrak{m}=N\mathcal{O}_K$ for an integer $N$ ($\geq2$)
whose prime factors are all inert in $K$, then
$K_\mathfrak{m}=K\left({x^{}_{K,\,n}}(\omega)\right)$ for every $n\in\mathbb{Z}_{\geq0}$.
\end{enumerate}
To this end, we shall make use of some inequalities
on special values of
the elliptic modular function and Siegel functions
(Lemmas \ref{j-inequalityLemma} and \ref{ginequality}),
rather than using $L$-function arguments adopted in
\cite{K-S-Y}, \cite{Ramachandra} and \cite{Schertz}.
\par
Lastly, we hope to utilize the above results (i), (ii) and (iii) to investigate
the images of (higher dimensional) Galois representations
attached to elliptic curves with complex multiplication.

\section {Fricke and Siegel functions}

In this preliminary section, we shall recall
the definitions and basic properties of Fricke and Siegel functions.
\par
Let $\mathbb{H}$ be the complex upper half-plane, that is,
$\mathbb{H}=\{\tau\in\mathbb{C}~|~\mathrm{Im}(\tau)>0\}$.
Let $j$ be the elliptic modular function on $\mathbb{H}$ given by
\begin{equation*}
j(\tau)=j([\tau,\,1])\quad(\tau\in\mathbb{H}),
\end{equation*}
where $[\tau,\,1]$ stands for the lattice $\mathbb{Z}\tau+\mathbb{Z}$ in $\mathbb{C}$
and $j([\tau,\,1])$ is the $j$-invariant of an
elliptic curve isomorphic to $\mathbb{C}/[\tau,\,1]$.
And, define the function $J$ on $\mathbb{H}$ by
\begin{equation*}
J(\tau)=\frac{1}{1728}\,j(\tau)\quad(\tau\in\mathbb{H}).
\end{equation*}
Furthermore, for $\mathbf{v}=\begin{bmatrix}v_1&v_2\end{bmatrix}\in M_{1,\,2}(\mathbb{Q})\setminus M_{1,\,2}(\mathbb{Z})$
we define the Fricke function $f_\mathbf{v}$ on $\mathbb{H}$ by
\begin{equation*}
f_\mathbf{v}(\tau)=-2^73^5\,\frac{{g^{}_{2}}(\tau){g^{}_{3}}(\tau)}{\Delta(\tau)}\,
\wp(v_1\tau+v_2;\,[\tau,\,1])
\quad(\tau\in\mathbb{H}),
\end{equation*}
where ${g^{}_{2}}(\tau)={g^{}_{2}}([\tau,\,1])$,
${g^{}_{3}}(\tau)={g^{}_{3}}([\tau,\,1])$ and
$\Delta(\tau)=\Delta([\tau,\,1])$.
Note that for $\mathbf{u},\,\mathbf{v}\in M_{1,\,2}(\mathbb{Q})
\setminus M_{1,\,2}(\mathbb{Z})$
\begin{equation}\label{fequal}
f_\mathbf{u}=f_\mathbf{v}\quad
\Longleftrightarrow\quad
\mathbf{u}\equiv\mathbf{v}~\textrm{or}~
-\mathbf{v}\Mod{M_{1,\,2}(\mathbb{Z})}
\end{equation}
(\cite[Lemma 10.4]{Cox}).
For a positive integer $N$, let $\mathcal{F}_N$
be the field given by
\begin{equation*}
\mathcal{F}_N=\left\{\begin{array}{ll}
\mathbb{Q}(j) & \textrm{if}~N=1,\\
\mathcal{F}_1\left(f_\mathbf{v}~|~\mathbf{v}\in
\frac{1}{N}M_{1,\,2}(\mathbb{Z})\setminus M_{1,\,2}(\mathbb{Z})\right)& \textrm{if}~N\geq2.
\end{array}\right.
\end{equation*}
Then, $\mathcal{F}_N$ is a Galois extension of $\mathcal{F}_1$ whose Galois group is isomorphic to
$\mathrm{GL}_2(\mathbb{Z}/N\mathbb{Z})/\langle -I_2\rangle$
(\cite[Theorem 6.6]{Shimura}).
And, it coincides with
the field of meromorphic modular functions of level $N$
whose Fourier coefficients
belong to the $N$th cyclotomic field (\cite[Proposition 6.9]{Shimura}).

\begin{proposition}\label{Fricke}
If $N\geq2$, $\mathbf{v}\in\frac{1}{N}M_{1,\,2}(\mathbb{Z})\setminus M_{1,\,2}
(\mathbb{Z})$ and $\gamma\in\mathrm{GL}_2(\mathbb{Z}/N\mathbb{Z})/\langle-I_2\rangle$, then
\begin{equation*}
f_\mathbf{v}^\gamma~=~f_{\mathbf{v}\gamma}.
\end{equation*}
Moreover, if $\gamma\in\mathrm{SL}_2(\mathbb{Z}/N\mathbb{Z})/\langle-I_2\rangle$, then
\begin{equation*}
f_\mathbf{v}^\gamma~=~f_\mathbf{v}\circ\alpha,
\end{equation*}
where $\alpha$ is any element of $\mathrm{SL}_2(\mathbb{Z})$
\textup{(}acting on $\mathbb{H}$ as fractional linear transformation\textup{)}
whose image in $\mathrm{SL}_2(\mathbb{Z}/N\mathbb{Z})/\langle-I_2\rangle$
is $\gamma$.
\end{proposition}
\begin{proof}
See \cite[Theorem 3 in Chapter 6]{Lang87} or \cite[Theorem 6.6]{Shimura}.
\end{proof}

For
$\mathbf{v}=\begin{bmatrix}v_1&v_2\end{bmatrix}\in M_{1,\,2}(\mathbb{Q})
\setminus M_{1,\,2}(\mathbb{Z})$, the Siegel function
$g_\mathbf{v}$ on $\mathbb{H}$
is given by the infinite product expansion
\begin{equation}\label{productexpansion}
g_\mathbf{v}(\tau)
=-e^{\pi\mathrm{i}v_2(v_1-1)}
q_\tau^{\frac{1}{2}(v_1^2-v_1+\frac{1}{6})}
(1-q_z)
\prod_{n=1}^\infty(1-q_\tau^nq_z)
(1-q_\tau^nq_z^{-1})\quad(\tau\in\mathbb{H}),
\end{equation}
where $q_\tau=e^{2\pi\mathrm{i}\tau}$ and
$q_z=e^{2\pi\mathrm{i}z}$ with $z=v_1\tau+v_2$.
Observe that $g_\mathbf{v}$ has neither a zero nor a pole on $\mathbb{H}$.

\begin{proposition}\label{Siegelfunction}
Let $N$ be an integer such that $N\geq2$,
and let $\mathbf{v}\in\frac{1}{N}M_{1,\,2}(\mathbb{Z})
\setminus M_{1,\,2}(\mathbb{Z})$.
\begin{enumerate}
\item[\textup{(i)}] If $\mathbf{u}
\in\frac{1}{N}M_{1,\,2}(\mathbb{Z})\setminus M_{1,\,2}(\mathbb{Z})$
such that $\mathbf{u}\equiv\mathbf{v}
~\textrm{or}~-\mathbf{v}\Mod{M_{1,\,2}(\mathbb{Z})}$,
then $g_\mathbf{u}^{12N}=g_\mathbf{v}^{12N}$.
\item[\textup{(ii)}] The function $g_\mathbf{v}^{12N}$ belongs to $\mathcal{F}_N$ and satisfies
\begin{equation*}
\left(g_\mathbf{v}^{12N}\right)^\gamma
~=~g_{\mathbf{v}\gamma}^{12N}
\quad(\gamma\in\mathrm{GL}_2(\mathbb{Z}/N\mathbb{Z})/\langle-I_2\rangle
\simeq\mathrm{Gal}(\mathcal{F}_N/\mathcal{F}_1)).
\end{equation*}
\end{enumerate}
\end{proposition}
\begin{proof}
\begin{enumerate}
\item[(i)] See \cite[Theorem 1.1 in Chapter 2 and p. 29]{K-L}.
\item[(ii)] See \cite[Theorem 1.2 and Proposition 1.3 in Chapter 2]{K-L}.
\end{enumerate}
\end{proof}

\begin{lemma}\label{ffgg}
Let $\mathbf{u},\,\mathbf{v}\in M_{1,\,2}(\mathbb{Q})
\setminus M_{1,\,2}(\mathbb{Z})$ such that
$\mathbf{u}\not\equiv\mathbf{v}~\textrm{and}~
-\mathbf{v}\Mod{M_{1,\,2}(\mathbb{Z})}$.
Then we have
\begin{equation*}
(f_\mathbf{u}-f_\mathbf{v})^6=
\frac{J^2(J-1)^3}{3^9}
\frac{g_{\mathbf{u}+\mathbf{v}}^6g_{\mathbf{u}-\mathbf{v}}^6}
{g_\mathbf{u}^{12}g_\mathbf{v}^{12}}.
\end{equation*}
\end{lemma}
\begin{proof}
See \cite[p. 51]{K-L}.
\end{proof}

\section {Extended form class groups}

In this section, we shall review some necessary consequences
of the theory of complex multiplication, and
introduce extended form class groups
which might be an extension of Gauss' form class group.
\par
Let $K$ be an imaginary quadratic field of discriminant $d_K$.
For a positive integer $N$, let $\mathcal{Q}_N(d_K)$ be the set of
primitive positive definite binary quadratic forms of discriminant $d_K$
whose leading coefficients
are relatively prime to $N$, that is,
\begin{equation*}
\mathcal{Q}_N(d_K)=
\left\{Q\left(\begin{bmatrix}x\\y\end{bmatrix}\right)=a_Qx^2+b_Qxy+c_Qy^2\in\mathbb{Z}[x,\,y]~\left|~
\begin{array}{l}
\gcd(a_Q,\,b_Q,\,c_Q)=1,\\
\gcd(a_Q,\,N)=1,\\
a_Q>0,\\
b_Q^2-4a_Qc_Q=d_K\\
\end{array}\right.
\right\}.
\end{equation*}
The congruence subgroup
\begin{equation*}
\Gamma_1(N)=\left\{
\gamma\in\mathrm{SL}_2(\mathbb{Z})~|~
\gamma\equiv\begin{bmatrix}1&\mathrm{*}\\
0&1\end{bmatrix}\Mod{N M_2(\mathbb{Z})}\right\}
\end{equation*}
defines an equivalence relation $\sim_N$ on the set $\mathcal{Q}_N(d_K)$ as
\begin{equation*}
Q\sim_N Q'\quad
\Longleftrightarrow\quad
Q'=Q\left(\gamma\begin{bmatrix}x\\y\end{bmatrix}\right)~
\textrm{for some}~\gamma\in\Gamma_1(N).
\end{equation*}
Let
\begin{equation*}
\mathrm{C}_N(d_K)=\mathcal{Q}_N(d_K)/\sim_N
\end{equation*}
be the set of equivalence classes.
For each $Q=a_Qx^2+b_Qxy+c_Qy^2\in\mathcal{Q}_N(d_K)$,
let $[Q]_N$ be its class in $\mathrm{C}_N(d_K)$,
and let
\begin{equation*}
{\tau^{}_{Q}}=\frac{-b_Q+\sqrt{d_K}}{2a_Q}
\end{equation*}
which is the zero of the quadratic polynomial $Q(x,\,1)$
lying in $\mathbb{H}$.
For a nontrivial ideal $\mathfrak{m}$ of $\mathcal{O}_K$,
let us denote by $\mathrm{Cl}(\mathfrak{m})$
the ray class group modulo $\mathfrak{m}$,
namely, $\mathrm{Cl}(\mathfrak{m})=I_K(\mathfrak{m})/P_{K,\,1}(\mathfrak{m})$
where $I_K(\mathfrak{m})$ is the group
of fractional ideals of $K$ relatively prime to $\mathfrak{m}$
and $P_{K,\,1}(\mathfrak{m})$
is the subgroup of $P_K(\mathfrak{m})$ (the subgroup of
$I_K(\mathfrak{m})$ consisting of principal fractional ideals) defined by
\begin{equation*}
P_{K,\,1}(\mathfrak{m})=\langle\nu\mathcal{O}_K~|~
\nu\in\mathcal{O}_K\setminus\{0\}~\textrm{such that}~
\nu\equiv1\Mod{\mathfrak{m}}
\rangle.
\end{equation*}
Then, when $\mathfrak{m}=N\mathcal{O}_{K}$, the map
\begin{eqnarray*}
\mathrm{C}_N(d_K)&\rightarrow&\mathrm{Cl}(N\mathcal{O}_K)\\
\mathrm{[}Q]_N & \mapsto& [[{\tau^{}_{Q}},\,1]]=[\mathbb{Z}\tau^{}_{Q}+\mathbb{Z}]
\end{eqnarray*}
is a well-defined bijection, through which we may regard
$\mathrm{C}_N(d_K)$ 
as a group isomorphic to $\mathrm{Cl}(N\mathcal{O}_K)$
(\cite[Theorem 2.9]{E-K-S}
or \cite[Theorem 2.5 and Proposition 5.3]{J-K-S20}).
The identity element of $\mathrm{C}_N(d_K)$
is the class $[Q_\mathrm{pr}]_N$ of
the principal form
\begin{equation*}
Q_\mathrm{pr}:=x^2+b_Kxy+c_Ky^2=
\left\{\begin{array}{ll}
\displaystyle x^2-\frac{d_K}{4}y^2 & \textrm{if}~d_K\equiv0\Mod{4},\\
\displaystyle x^2+xy+\frac{1-d_K}{4}y^2 & \textrm{if}~d_K\equiv1\Mod{4}.
\end{array}\right.
\end{equation*}
We call this group $\mathrm{C}_N(d_K)$
the extended form class group of discriminant $d_K$ and level $N$.
\par
In particular, $\mathrm{C}_1(d_K)$ is the classical
form class group of discriminant $d_K$,
originated and developed by Gauss (\cite{Gauss}) and Dirichlet (\cite{Dirichlet}).
A form $Q=a_Qx^2+b_Qxy+c_Qy^2$ in $\mathcal{Q}_1(d_K)$
is said to be reduced if
\begin{equation*}
-a_Q<b_Q\leq a_Q<c_Q\quad\textrm{or}\quad
0\leq b_Q\leq a_Q=c_Q.
\end{equation*}
And, this condition yields
\begin{equation}\label{boundaQ}
a_Q~\leq~\sqrt{\frac{|d_K|}{3}}.
\end{equation}
If we let $Q_1,\,Q_2,\,\ldots,\,Q_h$ be
all the reduced forms of discriminant $d_K$, then we have
$h=|\mathrm{C}_1(d_K)|$ and
\begin{equation}\label{C1}
\mathrm{C}_1(d_K)=\left\{[Q_1]_1,\,[Q_2]_1,\,\ldots,\,[Q_h]_1\right\}
\end{equation}
(\cite[Theorem 2.8]{Cox}).
Set
\begin{equation*}
\tau^{}_{K}=\left\{\begin{array}{cl}
\displaystyle\frac{\sqrt{d_K}}{2} & \textrm{if}~d_K\equiv0\Mod{4},\\
\displaystyle\frac{-1+\sqrt{d_K}}{2} & \textrm{if}~d_K\equiv1\Mod{4},
\end{array}\right.
\end{equation*}
and then $\tau^{}_{K}=\tau^{}_{Q_\mathrm{pr}}$ and
$\mathcal{O}_K=\left[\tau^{}_{K},\,1\right]$.
By the theory of complex multiplication,
we get the following results.

\begin{proposition}\label{CM}
Let $K$ be an imaginary quadratic field and $\mathfrak{m}$ be a nontrivial ideal of $\mathcal{O}_K$.
\begin{enumerate}
\item[\textup{(i)}] If $\mathfrak{m}=\mathcal{O}_K$, then we get
\begin{equation*}
K_\mathfrak{m}=H_K=K(j(\tau^{}_{K})).
\end{equation*}
Furthermore, if $Q_i$
\textup{(}$i=1,\,2,\,\ldots,\,h=|\mathrm{C}_{1}(d_{K})|$\textup{)}
are reduced forms of discriminant $d_K$, then the singular values
$j(\tau_{Q_i})$ are distinct \textup{(}Galois\textup{)} conjugates of $j(\tau^{}_{K})$ over $K$.
\item[\textup{(ii)}] If $\mathfrak{m}\neq\mathcal{O}_K$, then we have
\begin{equation*}
K_\mathfrak{m}=H_{K}\left(\mathfrak{h}^{}_{K}(\varphi^{}_{K}([\omega]))\right)
\end{equation*}
for any element $\omega$ of $K$ for which $[\omega]=\omega+\mathcal{O}_K$
is a generator of the $\mathcal{O}_K$-module $\mathfrak{m}^{-1}\mathcal{O}_K/\mathcal{O}_K$.
All such $\mathfrak{h}^{}_{K}(\varphi^{}_{K}([\omega]))$ are conjugate over $H_K$.
More precisely,
if $\xi_i$ \textup{(}$i=1,\,2,\,
\ldots,\,[K_\mathfrak{m}:H_K]$\textup{)}
are nonzero elements of $\mathcal{O}_K$ such that
\begin{equation*}
\left\{
(\xi_i)P_{K,\,1}(\mathfrak{m})~|~i=1,\,2,\,\ldots,\,
[K_\mathfrak{m}:H_K]\right\}=
P_K(\mathfrak{m})/P_{K,\,1}(\mathfrak{m})~(\simeq
\mathrm{Gal}(K_\mathfrak{m}/H_K)),
\end{equation*}
then
$\mathfrak{h}^{}_{K}(\varphi^{}_{K}([\xi_i\omega]))$ are all distinct conjugates of
$\mathfrak{h}^{}_{K}(\varphi^{}_{K}([\omega]))$ over $H_K$.
\end{enumerate}
\end{proposition}
\begin{proof}
\begin{enumerate}
\item[(i)] See \cite[Theorem 7.7 (ii)]{Cox} and \cite[Theorem 1 in Chapter 10]{Lang87}.
\item[(ii)] See \cite[Theorem 7 and its Corollary in Chapter 10]{Lang87}.
\end{enumerate}
\end{proof}

By modifying Shimura's reciprocity law (\cite[Theorem 6.31, Propositions 6.33
and 6.34]{Shimura}),
Eum, Koo and Shin
established the following proposition.

\begin{proposition}\label{CG}
Let $K$ be an imaginary quadratic field and
$N$ be a positive integer. Then the map
\begin{eqnarray*}
\mathrm{C}_N(d_K) &\rightarrow& \mathrm{Gal}(K_{(N)}/K)\\
\mathrm{[}Q]_N & \mapsto & \left(
f(\tau^{}_{K})\mapsto f^{\left[\begin{smallmatrix}
a_Q&(b_Q-b_K)/2\\0&1\end{smallmatrix}\right]}({\tau^{}_{Q}})~|~f\in\mathcal{F}_N~
\textrm{is finite at}~\tau^{}_{K}\right)
\end{eqnarray*}
is a well-defined isomorphism.
\end{proposition}
\begin{proof}
See \cite[Theorem 3.8]{E-K-S}.
\end{proof}

\begin{remark}\label{compatible}
If $M$ and $N$ are positive integers such that
$M\,|\,N$, then the natural map
\begin{eqnarray*}
\mathrm{C}_N(d_K) &\rightarrow& \mathrm{C}_M(d_K)\\
\mathrm{[}Q]_N & \mapsto & [Q]_M
\end{eqnarray*}
is a surjective homomorphism
(\cite[Remark 2.10 (i)]{E-K-S}).
\end{remark}

\section {Some applications of inequality on singular values of $j$}

Let $K$ be an imaginary quadratic field of discriminant $d_K$.
By using inequality argument on singular values of $j$
developed in \cite{J-K-S},
we shall show that coordinates of elliptic curves in the family $\{E_{K,\,n}\}_{n\in\mathbb{Z}_{\geq0}}$ described in (\ref{modelE}) can
be used in order to generate the ray class fields of $K$.
\par
Let $h_K$ denote the class number of $K$, i.e., $h_K=|\mathrm{C}_1(d_K)|=[H_K:K]$.
It is well known that
\begin{equation*}
h_K=1\quad\Longleftrightarrow\quad
d_K=-3,\,-4,\,-7,\,-8,\,-11,\,-19,\,-43,\,-67,\,-163
\end{equation*}
(\cite[Theorem 12.34]{Cox}).
So, if $h_K\geq2$, then we have $d_K\leq-15$.

\begin{lemma}\label{j-inequalityLemma}
If $h_K\geq2$ and $d_K\leq-20$, then we achieve
\begin{equation}\label{j-inequality}
\left|
\frac{J({\tau^{}_{Q}})^2(J({\tau^{}_{Q}})-1)^3}
{J(\tau^{}_{K})^2(J(\tau^{}_{K})-1)^3}
\right|~<~877383\,\left|q_{\tau^{}_{K}}\right|^\frac{5}{2}~(<1)
\end{equation}
for all nonprincipal reduced forms $Q$ of discriminant $d_K$.
\end{lemma}
\begin{proof}
See \cite[Lemma 6.3 (ii)]{J-K-S}.
\end{proof}

\begin{remark}\label{-15case}
If $d_K=-15$, then we get $\mathrm{C}_1(d_K)=\{[Q_1]_1,\,[Q_2]_1\}$ with
\begin{equation*}
Q_1=x^2+xy+4y^2\quad\textrm{and}\quad
Q_2=2x^2+xy+2y^2.
\end{equation*}
Moreover, we have
\begin{equation*}
j(\tau^{}_{K})=j(\tau^{}_{Q_1})=-52515-85995\frac{1+\sqrt{5}}{2}
\quad\textrm{and}\quad
j(\tau^{}_{Q_2})=-52515-85995\,\frac{1-\sqrt{5}}{2}
\end{equation*}
(\cite[Example 6.2.2]{Silverman}).
One can check that the inequality  (\ref{j-inequality}) also holds true.
\end{remark}

\begin{lemma}\label{HKc}
Let $K$ be an imaginary quadratic field
other than $\mathbb{Q}(\sqrt{-1})$ and $\mathbb{Q}(\sqrt{-3})$. Then we attain
\begin{equation*}
H_K=K(C_K^n)\quad\textrm{for every}~n\in\mathbb{Z}_{>0}.
\end{equation*}
\end{lemma}
\begin{proof}
If $h_K=1$, then the assertion is obvious because $H_K=K$. 
\par
Now, assume that $h_K\geq2$. Let $\sigma$ be an element of $\mathrm{Gal}(H_K/K)$ which 
fixes $C_K^n$. Then we find 
by Proposition \ref{CM} (i) that
\begin{equation*}
1=\left|
\frac{(C_K^n)^\sigma}{C_K^n}
\right|
=\left|\frac{\{J(\tau^{}_{K})^2(J(\tau^{}_{K})-1)^3\}^\sigma}{J(\tau^{}_{K})^2(J(\tau^{}_{K})-1)^3}\right|^n
=\left|\frac{J({\tau^{}_{Q}})^2(J({\tau^{}_{Q}})-1)^3}{J(\tau^{}_{K})^2(J(\tau^{}_{K})-1)^3}\right|^n
\end{equation*}
for some reduced form $Q$ of 
discriminant $d_K$.
Thus $Q$ must be $Q_\mathrm{pr}$ by Lemma \ref{j-inequalityLemma} and Remark \ref{-15case},
and hence $\sigma$ is the identity element of $\mathrm{Gal}(H_K/K)$
again by Proposition \ref{CM} (i). 
This observation implies by Galois theory that
$H_K$ is generated by $C_K^n$ over $K$. 
\end{proof}

\begin{theorem}\label{xy}
Let $K$ be an imaginary quadratic field other than $\mathbb{Q}(\sqrt{-1})$
and $\mathbb{Q}(\sqrt{-3})$, and let $\mathfrak{m}$
be a nontrivial proper integral ideal of $\mathcal{O}_K$.
Let $\omega$ be an element of $K$ such that
$\omega+\mathcal{O}_K$ is a generator of
the $\mathcal{O}_K$-module $\mathfrak{m}^{-1}/\mathcal{O}_K$.
If $K_\mathfrak{m}$ properly contains $H_K$, then we have
\begin{equation*}
K_\mathfrak{m}=K\left({x^{}_{K,\,n}}(\omega),\,{y^{}_{K,\,n}}(\omega)^2\right)
\quad
\textrm{for every $n\in\mathbb{Z}_{\geq0}$}.
\end{equation*}
\end{theorem}
\begin{proof}
For simplicity, let
\begin{equation*}
X={x^{}_{K,\,n}}(\omega)
=C_K^n\mathfrak{h}^{}_{K}(\varphi^{}_{K}([\omega]))\quad\textrm{and}\quad Y={y^{}_{K,\,n}}(\omega).
\end{equation*}
Set $L=K(X,\,Y^2)$ which is a subfield of $K_\mathfrak{m}$ by Proposition \ref{CM}
and the Weierstrass equation for $E_{K,\,n}$ stated in (\ref{modelE}).
\par
Suppose on the contrary that $K_\mathfrak{m}\neq L$.
Then there is a nonidentity element $\sigma$ of $\mathrm{Gal}(K_\mathfrak{m}/K)$ which
leaves both $X$ and $Y^2$ fixed.
Note further that
\begin{equation}\label{sigmanot}
\sigma\not\in\mathrm{Gal}(K_\mathfrak{m}/H_K)
\end{equation}
because $K_\mathfrak{m}=H_K(X)$ by Proposition \ref{CM} (ii).
By acting $\sigma$
on both sides of the equality
\begin{equation*}
Y^2=4X^3-AX-B
\quad\textrm{with}~A=\frac{J_K(J_K-1)}{27}C_K^{2n}\quad
\textrm{and}\quad B=\frac{J_K(J_K-1)^2}{27^2}C_K^{3n},
\end{equation*}
we obtain
\begin{equation*}
Y^2=4X^3-A^\sigma
X-B^\sigma.
\end{equation*}
It then follows that
\begin{equation}\label{AABB}
(A^\sigma-A)X=B-B^\sigma.
\end{equation}
Since
\begin{equation*}
AB=\frac{C_K^{5n+1}}{27^3}
\end{equation*}
generates $H_K$ over $K$
by Lemma \ref{HKc}, we deduce by (\ref{sigmanot}) and (\ref{AABB}) that
$A^\sigma-A\neq0$ and
\begin{equation*}
X=-\frac{B^\sigma-B}{A^\sigma-A}\in H_K.
\end{equation*}
Then we get
\begin{equation*}
H_K=H_K(X)=K_\mathfrak{m},
\end{equation*}
which contradicts the hypothesis $K_\mathfrak{m}\supsetneq H_K$.
\par
Therefore we conclude that
\begin{equation*}
K_\mathfrak{m}=L=K\left({x^{}_{K,\,n}}(\omega),\,{y^{}_{K,\,n}}(\omega)^2\right).
\end{equation*}

\end{proof}

\begin{proposition}\label{x-generator}
Let $K$ be an imaginary quadratic field other than $\mathbb{Q}(\sqrt{-1})$
and $\mathbb{Q}(\sqrt{-3})$, and let $\mathfrak{m}$
be a nontrivial proper integral ideal of $\mathcal{O}_K$.
Let $\omega$ be an element of $K$ such that
$\omega+\mathcal{O}_K$ is a generator of
the $\mathcal{O}_K$-module $\mathfrak{m}^{-1}/\mathcal{O}_K$.
Then we have
\begin{equation*}
K_\mathfrak{m}=K\left({x^{}_{K,\,n}}(\omega)\right)\quad\textrm{for sufficiently large $n\in\mathbb{Z}_{\geq0}$}.
\end{equation*}
\end{proposition}
\begin{proof}
Note that
$C_K=J(\tau^{}_{K})^2(J(\tau^{}_{K})-1)^3$ is nonzero
because $K$ is different from $\mathbb{Q}(\sqrt{-1})$
and $\mathbb{Q}(\sqrt{-3})$.
There are two possible cases\,: $h_K=1$ or $h_K\geq2$.
\begin{enumerate}
\item[Case 1.] If $h_K=1$ (and so $H_K=K$), then for any $n\in\mathbb{Z}_{\geq0}$
\begin{eqnarray*}
K\left({x^{}_{K,\,n}}(\omega)\right)&=&H_K\left(C_K^n\mathfrak{h}^{}_{K}(\varphi^{}_{K}([\omega]))\right)
\quad\textrm{by (\ref{xchp})}\\
&=&H_K\left(\mathfrak{h}^{}_{K}(\varphi^{}_{K}([\omega]))\right)\quad\textrm{by Proposition \ref{CM} (i)}\\
&=&K_\mathfrak{m}\quad\textrm{by Proposition \ref{CM} (ii)}.
\end{eqnarray*}
\item[Case 2.] Consider the case where $h_K\geq2$.
Let $\mathrm{Gal}(H_K/K)=\{\sigma_1=\mathrm{id},\,\sigma_2,\,
\ldots,\,\sigma_{h_K}\}$ and $d=[K_\mathfrak{m}:H_K]$.
Observe by Proposition \ref{CM} (i) that for each $i=1,\,2,\,\ldots,\,h_K$
there is a unique reduced form $Q_i$ of discriminant $d_K$, and so
$J(\tau^{}_{K})^{\sigma_i}=J(\tau^{}_{Q_i})$.
By Lemma \ref{j-inequalityLemma} and Remark \ref{-15case}
we can take a sufficiently large positive integer $m$ so that
\begin{equation*}
\left|\frac{C_K^{\sigma_i}}{C_K}
\right|^{md}=
\left|
\frac{J(\tau^{}_{Q_i})^2(J(\tau^{}_{Q_i})-1)^3}
{J(\tau^{}_{K})^2(J(\tau^{}_{K})-1)^3}
\right|^{md}<\left|
\frac{\mathrm{N}_{K_\mathfrak{m}/H_K}\left(\mathfrak{h}^{}_{K}(\varphi^{}_{K}([\omega]))\right)}
{\mathrm{N}_{K_\mathfrak{m}/H_K}\left(\mathfrak{h}^{}_{K}(\varphi^{}_{K}([\omega]))\right)^{\sigma_i}}\right|
\quad\textrm{for all}~i=2,\,3,\,\ldots,\,h_K.
\end{equation*}
We then see by (\ref{xchp}) and
Proposition \ref{CM} (i) that if $n\in\mathbb{Z}_{\geq0}$
satisfies $n\geq m$ and $2\leq i\leq h_K$, then
\begin{equation*}
\left|
\frac{\mathrm{N}_{K_\mathfrak{m}/H_K}({x^{}_{K,\,n}}(\omega))^{\sigma_i}}
{\mathrm{N}_{K_\mathfrak{m}/H_K}({x^{}_{K,\,n}}(\omega))}
\right|=\left|
\frac{C_K^{\sigma_i}}{C_K}\right|^{nd}
\left|
\frac{\mathrm{N}_{K_\mathfrak{m}/H_K}\left(\mathfrak{h}^{}_{K}(\varphi^{}_{K}([\omega]))\right)^{\sigma_i}}
{\mathrm{N}_{K_\mathfrak{m}/H_K}\left(\mathfrak{h}^{}_{K}(\varphi^{}_{K}([\omega]))\right)}\right|<1.
\end{equation*}
This observation implies that
\begin{equation}\label{KNH}
K\left(\mathrm{N}_{K_\mathfrak{m}/H_K}({x^{}_{K,\,n}}(\omega))\right)=H_K.
\end{equation}
Hence we derive that if $n\geq m$, then
\begin{eqnarray*}
K\left({x^{}_{K,\,n}}(\omega)\right)&=&
K\left({x^{}_{K,\,n}}(\omega),\,\mathrm{N}_{K_\mathfrak{m}/H_K}({x^{}_{K,\,n}}(\omega))\right)\\
&&\quad\quad\quad\textrm{since $K({x^{}_{K,\,n}}(\omega))$ ($\subseteq K_\mathfrak{m}$)
is an abelian extension of $K$}\\
&=&H_K\left({x^{}_{K,\,n}}(\omega)\right)\quad\textrm{by (\ref{KNH})}\\
&=&K_\mathfrak{m}\quad\textrm{by (\ref{xchp}) and Proposition \ref{CM}}.
\end{eqnarray*}
\end{enumerate}
\end{proof}

\section {Generation of ray class fields by $x$-coordinates of elliptic curves}

By using some interesting inequalities
on special values of Siegel functions,
we shall find a concrete bound of $n$ in Proposition \ref{x-generator}
for which if $n$ is greater than or equal to this, then
${x^{}_{K,\,n}}(\omega)$ generates $K_\mathfrak{m}$ over $K$.

\begin{lemma}\label{ginequality}
Let $\mathbf{v}\in M_{1,\,2}(\mathbb{Q})\setminus M_{1,\,2}(\mathbb{Z})$, and 
let $\tau\in\mathbb{H}$ such that $|q_\tau|=|e^{2\pi\mathrm{i}\tau}|\leq e^{-\pi\sqrt{3}}$.
\begin{enumerate}
\item[\textup{(i)}] We have
\begin{equation*}
\left|g_\mathbf{v}(\tau)\right|~<~
2.29\,|q_\tau|^{-\frac{1}{24}}.
\end{equation*}
\item[\textup{(ii)}] If $\mathbf{v}\in\frac{1}{N}M_{1,\,2}(\mathbb{Z})\setminus M_{1,\,2}(\mathbb{Z})$
for an integer $N\geq2$, then we obtain
\begin{equation*}
\left|g_\mathbf{v}(\tau)\right|~>~\frac{0.76|q_\tau|^{\frac{1}{12}}}{N}.
\end{equation*}
\end{enumerate}
\end{lemma}
\begin{proof}
Let $\mathbf{v}=\begin{bmatrix}v_1&v_2\end{bmatrix}$
and $z=v_1\tau+v_2$. By Proposition \ref{Siegelfunction} (i)
we may assume that
$0\leq v_1\leq\frac{1}{2}$. Set $s=|q_\tau|$.
\begin{enumerate}
\item[(i)] We then derive that
\begin{eqnarray*}
|g_\mathbf{v}(\tau)|&\leq&
|q_\tau|^{\frac{1}{2}(v_1^2-v_1+\frac{1}{6})}
(1+|q_z|)\prod_{n=1}^\infty(1+|q_\tau|^n|q_z|)(1+|q_\tau|^n|q_z|^{-1})
\quad\textrm{by (\ref{productexpansion})}\\
&=&s^{\frac{1}{2}(v_1^2-v_1+\frac{1}{6})}(1+s^{v_1})\prod_{n=1}^\infty
(1+s^{n+v_1})(1+s^{n-v_1})\\
&\leq&s^{-\frac{1}{24}}(1+1)\prod_{n=1}^\infty(1+s^{n-\frac{1}{2}})^2
\quad\textrm{since}~0\leq v_1\leq\frac{1}{2}~\textrm{and}~v_1^2-v_1+\frac{1}{6}\geq-\frac{1}{12}\\
&\leq&2s^{-\frac{1}{24}}\prod_{n=1}^\infty e^{2(e^{-\pi\sqrt{3}})^{n-\frac{1}{2}}}
\quad
\textrm{because}~1+x<e^x~\textrm{for}~x>0\\
&=&2s^{-\frac{1}{24}}e^{2\sum_{n=1}^\infty(e^{-\pi\sqrt{3}})^{n-\frac{1}{2}}}\\
&=&2s^{-\frac{1}{24}}e^{\frac{2e^{-\frac{\pi\sqrt{3}}{2}}}{1-e^{-\pi\sqrt{3}}}}\\
&<&2.29\,s^{-\frac{1}{24}}.
\end{eqnarray*}
\item[(ii)] Furthermore, we see that
\begin{eqnarray*}
|g_\mathbf{v}(\tau)|&\geq&
s^{\frac{1}{2}(v_1^2-v_1+\frac{1}{6})}
|1-q_z|\prod_{n=1}^\infty(1-s^{n+v_1})(1-s^{n-v_1})
\quad\textrm{by (\ref{productexpansion})}\\
&\geq&
s^{\frac{1}{12}}
\min\left\{|1-e^\frac{2\pi\mathrm{i}}{N}|,\,1-s^\frac{1}{N}\right\}\prod_{n=1}^\infty(1-s^{n-\frac{1}{2}})^2
\quad\textrm{because}~v_1^2-v_1+\frac{1}{6}\leq\frac{1}{6}\\
&\geq&s^{\frac{1}{12}}\min\left\{\sin\frac{\pi}{N},\,
1-(e^{-\pi\sqrt{3}})^\frac{1}{N}\right\}
\prod_{n=1}^\infty e^{-4(e^{-\pi\sqrt{3}})^{n-\frac{1}{2}}}\\
&&\hspace{4.5cm}\textrm{since}~
1-x>e^{-2x}~\textrm{for}~0<x<\frac{1}{2}\\
&>&s^{\frac{1}{12}}\frac{1}{N}e^{-4\sum_{n=1}^\infty
(e^{-\pi\sqrt{3}})^{n-\frac{1}{2}}}\quad\textrm{because both}~
\sin(\pi x)~\textrm{and}~\,1-e^{-\pi\sqrt{3}x}~\textrm{are}>x\\
&&\hspace{4.5cm} \textrm{for}~0<x\leq\frac{1}{2}\\
&=&s^\frac{1}{12}\frac{1}{N}e^{\frac{-4e^{-\frac{\pi\sqrt{3}}{2}}}{1-e^{-\pi\sqrt{3}}}}\\
&>&\frac{0.76s^{\frac{1}{12}}}{N}.
\end{eqnarray*}
\end{enumerate}
\end{proof}

\begin{theorem}\label{conditional}
Let $K$ be an imaginary quadratic field other than $\mathbb{Q}(\sqrt{-1})$
and $\mathbb{Q}(\sqrt{-3})$.
Let $\mathfrak{m}$ be a proper nontrivial ideal of $\mathcal{O}_K$
in which $N_\mathfrak{m}$ \textup{($\geq2$)} is the least positive integer.
Let $\omega$ be an element of $K$
such that $\omega+\mathcal{O}_K$ is a generator of the
$\mathcal{O}_K$-module $\mathfrak{m}^{-1}/\mathcal{O}_K$.
If $K_\mathfrak{m}$ properly contains $H_K$, then
\begin{equation*}
K_\mathfrak{m}=K({x^{}_{K,\,n}}(\omega))
\end{equation*}
for every nonnegative integer $n$ satisfying
\begin{equation}\label{lowerbound}
n\geq\frac{\frac{13}{24}\pi\sqrt{|d_K|}+
6\ln\left(\frac{229}{76}N_\mathfrak{m}\right)}{\frac{5}{2}\pi\sqrt{|d_K|}-\ln 877383}-\frac{1}{6}.
\end{equation}
\end{theorem}
\begin{proof}
Since $N_\mathfrak{m}\mathcal{O}_K\subseteq\mathfrak{m}$
and $\omega\in\mathfrak{m}^{-1}\setminus\mathcal{O}_K$, we have
\begin{equation}\label{Nab}
N_\mathfrak{m}\omega=a\tau^{}_{K}+b\quad\textrm{for some}~a,\,b\in\mathbb{Z}~
\textrm{such that}~\begin{bmatrix}a&b\end{bmatrix}
\not\in N_\mathfrak{m}M_{1,\,2}(\mathbb{Z}).
\end{equation}
Let $n$ be a nonnegative integer satisfying (\ref{lowerbound}).
If $h_K=1$, then the assertion holds by the proof (Case 1) of
Proposition \ref{x-generator}.
\par
Now, we assume $h_K\geq2$.
Since $K_\mathfrak{m}$ properly
contains $H_K$,
one can take a nonidentity element $\rho$
of $\mathrm{Gal}(K_\mathfrak{m}/H_K)$.
Note that $\rho$ does not fix ${x^{}_{K,\,n}}(\omega)$
due to the fact $K_\mathfrak{m}=H_K({x^{}_{K,\,n}}(\omega))$ by
(\ref{xchp}) and
Proposition \ref{CM} (ii).
Suppose on the contrary that
${x^{}_{K,\,n}}(\omega)$ does not generate $K_\mathfrak{m}$ over $K$.
Then there exists at least one nonidentity element $\sigma$ in
$\mathrm{Gal}(K_\mathfrak{m}/K({x^{}_{K,\,n}}(\omega)))
=\mathrm{Gal}(H_K({x^{}_{K,\,n}}(\omega))/
K({x^{}_{K,\,n}}(\omega)))$.
Let $P$ be a quadratic form in $\mathcal{Q}_{N_\mathfrak{m}}(d_K)$
such that $[P]_{N_\mathfrak{m}}$ maps to $\sigma$ through the surjection
\begin{equation*}
\mathrm{C}_{N_\mathfrak{m}}(d_K)
\stackrel{\sim}{\longrightarrow}\mathrm{Gal}(K_{(N_\mathfrak{m})}/K)
\stackrel{\textrm{restriction}}{\longrightarrow}\mathrm{Gal}(K_\mathfrak{m}/K)
\end{equation*}
whose first map is the isomorphism
given in Proposition \ref{CG}. 
It follows from (\ref{C1}), Proposition \ref{CG} and Remark \ref{compatible} that
\begin{equation*}
P=Q^\gamma\quad\textrm{for some
nonpricipal
reduced form $Q$ and}~
\gamma\in\mathrm{SL}_2(\mathbb{Z}).
\end{equation*}
Here we observe that 
\begin{equation}\label{PQ}
\tau^{}_P=\gamma^{-1}({\tau^{}_{Q}})\quad\textrm{and}\quad a_Q\geq2. 
\end{equation}
Then we deduce that
\begin{eqnarray*}
1&=&\left|
\frac{\left({x^{}_{K,\,n}}(\omega)-{x^{}_{K,\,n}}(\omega)^\rho\right)^\sigma}
{{x^{}_{K,\,n}}(\omega)-{x^{}_{K,\,n}}(\omega)^\rho}
\right|\\
&&\hspace{1cm}\textrm{because $\sigma$ is the identity on
$K({x^{}_{K,\,n}}(\omega))$ which contains}~{x^{}_{K,\,n}}(\omega)^\rho\\
&=&\left|
\frac{\left(C_K^nf_\mathbf{u}(\tau^{}_K)-(C_K^nf_\mathbf{u}(\tau^{}_K))^\rho\right)^\sigma}
{C_K^nf_\mathbf{u}(\tau^{}_K)-(C_K^nf_\mathbf{u}(\tau^{}_K))^\rho}
\right|\quad\textrm{with}~
\mathbf{u}=\begin{bmatrix}
\frac{a}{N_\mathfrak{m}} & \frac{b}{N_\mathfrak{m}}\end{bmatrix}\\
&&\hspace{1cm}\textrm{by (\ref{xchp}), (\ref{Nab}) and the definitions of
$\mathfrak{h}^{}_{K}$, $\varphi^{}_{K}$ and a Fricke function}\\
&=&\left|
\frac{\{J(\tau^{}_{K})^{2n}(J(\tau^{}_{K})-1)^{3n}
(f_\mathbf{u}(\tau^{}_{K})-
f_\mathbf{v}(\tau^{}_{K}))\}^\sigma}
{J(\tau^{}_{K})^{2n}(J(\tau^{}_{K})-1)^{3n}
(f_\mathbf{u}(\tau^{}_{K})-
f_\mathbf{v}(\tau^{}_{K}))}\right|\\
&&\hspace{1cm}\textrm{for some}~\mathbf{v}\in\frac{1}{N_\mathfrak{m}}
M_{1,\,2}(\mathbb{Z})\setminus M_{1,\,2}(\mathbb{Z})~
\textrm{such that}~\mathbf{u}\not\equiv\mathbf{v}~\textrm{and}\,-\mathbf{v}
\Mod{M_{1,\,2}(\mathbb{Z})}\\
&&\hspace{1cm}
\textrm{by Proposition \ref{CM} (ii) and (\ref{fequal})
since}~\rho\in\mathrm{Gal}(K_\mathfrak{m}/H_K)
\setminus\{\mathrm{id}_{K_\mathfrak{m}}\}\\
&=&\left|\frac{J(\tau^{}_P)^2(J(\tau^{}_P)-1)^3}{J(\tau^{}_{K})^2(J(\tau^{}_{K})-1)^3}
\right|^n
\left|
\frac{f_{\mathbf{u}'}(\tau^{}_P)-
f_{\mathbf{v}'}(\tau^{}_P)}
{f_\mathbf{u}(\tau^{}_{K})-
f_\mathbf{v}(\tau^{}_{K})}
\right|\\
&&\hspace{1cm}\textrm{for some}~\mathbf{u}',\,\mathbf{v}'\in
\frac{1}{N_\mathfrak{m}}M_{1,\,2}(\mathbb{Z})\setminus M_{1,\,2}(\mathbb{Z})~
\textrm{such that}~\mathbf{u}'\not\equiv\mathbf{v}'~\textrm{and}\,\,-\mathbf{v}'\Mod{M_{1,\,2}(\mathbb{Z})}\\
&&\hspace{1cm}\textrm{by Proposition \ref{CG}}\\
&=&\left|\frac{J(\gamma^{-1}({\tau^{}_{Q}}))^2(J(\gamma^{-1}({\tau^{}_{Q}}))-1)^3}{J(\tau^{}_{K})^2(J(\tau^{}_{K})-1)^3}
\right|^n
\left|
\frac{f_{\mathbf{u}'}(\gamma^{-1}({\tau^{}_{Q}}))-
f_{\mathbf{v}'}(\gamma^{-1}({\tau^{}_{Q}}))}
{f_\mathbf{u}(\tau^{}_{K})-
f_\mathbf{v}(\tau^{}_{K})}
\right|\quad\textrm{by (\ref{PQ})}\\
&=&\left|\frac{J({\tau^{}_{Q}})^2(J({\tau^{}_{Q}})-1)^3}{J(\tau^{}_{K})^2(J(\tau^{}_{K})-1)^3}
\right|^n
\left|
\frac{f_{\mathbf{u}''}({\tau^{}_{Q}})-
f_{\mathbf{v}''}({\tau^{}_{Q}})}
{f_\mathbf{u}(\tau^{}_{K})-
f_\mathbf{v}(\tau^{}_{K})}
\right|\quad\textrm{with}~\mathbf{u}''=\mathbf{u}'\gamma^{-1}~
\textrm{and}~\mathbf{v}''=\mathbf{v}'\gamma^{-1}\\
&&\hspace{1cm}\textrm{by the fact $J\in\mathcal{F}_1$ and Proposition \ref{Fricke}}\\
&=&\left|\frac{J({\tau^{}_{Q}})^2(J({\tau^{}_{Q}})-1)^3}{J(\tau^{}_{K})^2(J(\tau^{}_{K})-1)^3}
\right|^{n+\frac{1}{6}}
\left|
\frac{g_{\mathbf{u}''+\mathbf{v}''}({\tau^{}_{Q}})g_{\mathbf{u}''-\mathbf{v}''}({\tau^{}_{Q}})
g_\mathbf{u}(\tau^{}_{K})^2g_\mathbf{v}(\tau^{}_{K})^2}
{g_{\mathbf{u}+\mathbf{v}}(\tau^{}_{K})g_{\mathbf{u}-\mathbf{v}}(\tau^{}_{K})
g_{\mathbf{u}''}({\tau^{}_{Q}})^2g_{\mathbf{v}''}({\tau^{}_{Q}})^2}
\right|\quad\textrm{by Lemma \ref{ffgg}}\\
&<&\left(877383\,|q_{\tau^{}_{K}}|^\frac{5}{2}\right)^{n+\frac{1}{6}}
\frac{2.29^6|q_{{\tau^{}_{Q}}}|^{-\frac{1}{12}}|q_{\tau^{}_{K}}|^{-\frac{1}{6}}}
{\left(\frac{0.76}{N_\mathfrak{m}}\right)^6|q_{\tau^{}_{K}}|^\frac{1}{6}|q_{{\tau^{}_{Q}}}|^\frac{1}{3}}
\quad\textrm{by Lemmas \ref{j-inequalityLemma}, \ref{ginequality} and Remark \ref{-15case}}\\
&&\hspace{1cm}\textrm{because}~
\left|q_{\tau^{}_{K}}\right|=e^{-\pi\sqrt{|d_K|}}\leq e^{-\pi\sqrt{15}}~\textrm{and}~
\big|q_{\tau^{}_{Q}}\big|=e^{-\frac{\pi\sqrt{|d_K|}}{a_Q}}\leq e^{-\pi\sqrt{3}}~
\textrm{by (\ref{boundaQ})}\\
&=&\left(877383\,|q_{\tau^{}_{K}}|^\frac{5}{2}\right)^{n+\frac{1}{6}}
\left(\frac{229}{76}N_\mathfrak{m}\right)^6|q_{{\tau^{}_{Q}}}|^{-\frac{5}{12}}|q_{\tau^{}_{K}}|^{-\frac{1}{3}}\\
&\leq&\left(877383\,|q_{\tau^{}_{K}}|^\frac{5}{2}\right)^{n+\frac{1}{6}}
\left(\frac{229}{76}N_\mathfrak{m}\right)^6|q_{\tau^{}_{K}}|^{-\frac{5}{24}}|q_{\tau^{}_{K}}|^{-\frac{1}{3}}\\
&&\hspace{1cm}\textrm{since}~\big|q_{\tau^{}_{K}}\big|^{-1}>~1~\textrm{and}~
\big|q_{{\tau^{}_{Q}}}\big|^{-1}
=\left(|q_{\tau^{}_{K}}|^{-1}\right)^\frac{1}{a_Q}\leq\left(|q_{\tau^{}_{K}}|^{-1}\right)^\frac{1}{2}~\textrm{by (\ref{PQ})}\\
&=&\left(877383\,e^{-\frac{5}{2}\pi\sqrt{|d_K|}}\right)^{n+\frac{1}{6}}
\left(\frac{229}{76}N_\mathfrak{m}\right)^6e^{\frac{13}{24}\pi\sqrt{|d_K|}}.
\end{eqnarray*}
Now, by taking logarithm we obtain the inequality
\begin{equation*}
0~<~\left(n+\frac{1}{6}\right)
\left(\ln 877383-\frac{5}{2}\pi\sqrt{|d_K|}\right)
+6\ln\left(\frac{229}{76}N_\mathfrak{m}\right)+\frac{13}{24}\pi\sqrt{|d_K|}
\end{equation*}
with
\begin{equation*}
\ln 877383-\frac{5}{2}\pi\sqrt{|d_K|}~\leq~
\ln 877383-\frac{5}{2}\pi\sqrt{15}~<~0.
\end{equation*}
But this contradicts (\ref{lowerbound}).
Therefore we conclude that $K_\mathfrak{m}=K\left({x^{}_{K,\,n}}(\omega)\right)$.
\end{proof}

\begin{example}
Let $K=\mathbb{Q}(\sqrt{-5})$, and so $d_K=-20$. 
Note that
\begin{equation*}
\left\{
\begin{array}{ll}
\textrm{$2$ is ramified in $K$} & \textrm{since}~2\,|\,d_K,\\
\textrm{$13$ is inert in $K$} & \textrm{due to}~(\frac{d_K}{13})=-1,\\
\textrm{$23$ splits completely in $K$} & \textrm{because}~
(\frac{d_K}{23})=1
\end{array}
\right.
\end{equation*}
(\cite[Proposition 5.16]{Cox}). And, let $\mathfrak{m}
=\mathfrak{p}^{}_1\mathfrak{p}^{}_2
\mathfrak{p}^{}_3$
be the product of three prime ideals
\begin{equation*}
\mathfrak{p}^{}_1=[1+\sqrt{-5},\,2],
\quad\mathfrak{p}^{}_2=[13\sqrt{-5},\,13]
\quad\textrm{and}\quad
\mathfrak{p}^{}_3=[15+\sqrt{-5},\,23]
\end{equation*}
of $\mathcal{O}_K$ satisfying
$2\mathcal{O}_K=\mathfrak{p}_1^2$,
$13\mathcal{O}_K=\mathfrak{p}^{}_2$
and $23\mathcal{O}_K=\mathfrak{p}^{}_3\overline{\mathfrak{p}}^{}_3$. In this case, by checking the degree formula for $[K_{\mathfrak{m}}:H_{K}]$ we see that $K_{\mathfrak{m}}$ properly contains $H_{K}$.
Since 
\begin{equation*}
\mathfrak{m}=\mathfrak{p}^{}_1\mathfrak{p}^{}_2
\mathfrak{p}^{}_3\supset
\mathfrak{p}_1^2\mathfrak{p}^{}_2\mathfrak{p}^{}_3
\overline{\mathfrak{p}}^{}_3=(2\cdot13\cdot23)\mathcal{O}_K,
\end{equation*}
we get $N_\mathfrak{m}=2\cdot13\cdot23=598$, and hence one can estimate
\begin{equation*}
\frac{\frac{13}{24}\pi\sqrt{|d_K|}+
6\ln\left(\frac{229}{76}N_\mathfrak{m}\right)}{\frac{5}{2}\pi\sqrt{|d_K|}-\ln 877383}-\frac{1}{6}
~=~\frac{\frac{13}{24}\pi\sqrt{20}+
6\ln\left(\frac{229}{76}\cdot598\right)}
{\frac{5}{2}\pi\sqrt{20}-\ln 877383}-\frac{1}{6}
~\approx~2.286282.
\end{equation*}
If $\omega$ is an element of $K$
such that $\omega+\mathcal{O}_K$ is a generator of the
$\mathcal{O}_K$-module $\mathfrak{m}^{-1}/\mathcal{O}_K$, then
we obtain by Theorem \ref{conditional} that
\begin{equation*}
K_\mathfrak{m}=K({x^{}_{K,\,n}}(\omega))
\quad\textrm{for all}~n\geq3. 
\end{equation*}
\end{example}

\begin{remark}
At this stage we conjecture that
Theorem \ref{conditional} may hold
for every $n\in\mathbb{Z}_{\geq0}$.
\end{remark}

\section {Ray class fields of special moduli}

Let $K$ be an imaginary quadratic field other than $\mathbb{Q}(\sqrt{-1})$
and $\mathbb{Q}(\sqrt{-3})$,
and let $N$ ($\geq2$) be an integer
whose prime factors are all inert in $K$.
In this last section,
we shall consider the special
case where $\mathfrak{m}=N\mathcal{O}_K$
and show that Theorem \ref{conditional}
is also true for every $n\in\mathbb{Z}_{\geq0}$.

\begin{lemma}\label{modularunit1}
Let $f\in\mathcal{F}_1$. If $f$ has neither a zero nor a pole on $\mathbb{H}$, then it is a
nonzero rational number.
\end{lemma}
\begin{proof}
See \cite[Lemma 2.1]{K-S} and \cite[Theorem 2 in Chapter 5]{Lang87}.
\end{proof}

For an integer $N\geq2$, let
\begin{equation*}
S_N=\left\{
\begin{bmatrix}s&t\end{bmatrix}\in M_{1,\,2}(\mathbb{Z})~|~0\leq s,\,t<N~\textrm{and}~
\gcd(N,\,s,\,t)=1\right\}.
\end{equation*}
We define an equivalence relation $\equiv_N$ on the set $S_N$ as follows\,:
for $\mathbf{u},\,\mathbf{v}\in S_N$
\begin{equation*}
\mathbf{u}\equiv_N\mathbf{v}\quad
\Longleftrightarrow\quad
\mathbf{u}\equiv\mathbf{v}~\textrm{or}~-\mathbf{v}\Mod{NM_{1,\,2}(\mathbb{Z})}.
\end{equation*}
Let
\begin{equation*}
P_N=\left\{(\mathbf{u},\,\mathbf{v})~|~\mathbf{u},\,\mathbf{v}\in S_N~
\textrm{such that}~\mathbf{u}\not\equiv_N\mathbf{v}\right\}
\quad\textrm{and}\quad m_N=|P_N|.
\end{equation*}
Since $\begin{bmatrix}1&0\end{bmatrix},
\,\begin{bmatrix}0&1\end{bmatrix}$ represent distinct classes in $S_N/\equiv_N$, we claim $m_N\geq2$.

\begin{lemma}\label{normconstant}
If $N$ is an integer such that $N\geq2$, then we have
\begin{equation*}
\prod_{(\mathbf{u},\,\mathbf{v})\in P_N}
\left(f_{\frac{1}{N}\mathbf{u}}-f_{\frac{1}{N}\mathbf{v}}\right)^6=
k\left\{J^2(J-1)^3\right\}^{m_N}\quad\textrm{for some}~k\in\mathbb{Q}\setminus\{0\}.
\end{equation*}
\end{lemma}
\begin{proof}
For $\begin{bmatrix}a&b\end{bmatrix}\in M_{1,\,2}(\mathbb{Z})$
with $\gcd(N,\,a,\,b)=1$, let ${\pi^{}_N}\left(\begin{bmatrix}a&b\end{bmatrix}\right)$
denote the unique element of $S_N$ satisfying
\begin{equation*}
{\pi^{}_N}\left(\begin{bmatrix}a&b\end{bmatrix}\right)
\equiv\begin{bmatrix}a&b\end{bmatrix}\Mod{NM_{1,\,2}(\mathbb{Z})}.
\end{equation*}
Let $\alpha\in M_{2}(\mathbb{Z})$ with
$\gcd(N,\,\det(\alpha))=1$, and let
$\widetilde{\alpha}$ be its
image in $\mathrm{GL}_2(\mathbb{Z}/N\mathbb{Z})/\langle-I_2\rangle$
($\simeq\mathrm{Gal}(\mathcal{F}_N/\mathcal{F}_1)$).
Setting
\begin{equation*}
f=\prod_{(\mathbf{u},\,\mathbf{v})\in P_N}
\left(f_{\frac{1}{N}\mathbf{u}}-f_{\frac{1}{N}\mathbf{v}}\right)^6,
\end{equation*}
we find that
\begin{eqnarray*}
f^{\widetilde{\alpha}}&=&
\prod_{(\mathbf{u},\,\mathbf{v})\in P_N}
\left(f_{\frac{1}{N}\mathbf{u}\widetilde{\alpha}}-f_{\frac{1}{N}\mathbf{v}\widetilde{\alpha}}\right)^6\quad\textrm{by
Proposition \ref{Fricke}}\\
&=&
\prod_{(\mathbf{u},\,\mathbf{v})\in P_N}
\left(f_{\frac{1}{N}{\pi^{}_N}(\mathbf{u}\alpha)}-
f_{\frac{1}{N}{\pi^{}_N}(\mathbf{v}\alpha)}\right)^6\quad\textrm{by
(\ref{fequal})}\\
&=&f
\end{eqnarray*}
because the mapping $S_N\rightarrow S_N$, $\mathbf{u}\mapsto{\pi^{}_N}(\mathbf{u}\alpha)$,
gives rise to an injection (and so, a bijection) of $P_N$ into itself.
This observation implies by Galois theory that $f$ lies in $\mathcal{F}_1$.
\par
On the other hand, we attain
\begin{eqnarray*}
f&=&\prod_{(\mathbf{u},\,\mathbf{v})\in P_N}
\frac{J^2(J-1)^3}{3^9}\frac{g_{\frac{1}{N}(\mathbf{u}+\mathbf{v})}^6
g_{\frac{1}{N}(\mathbf{u}-\mathbf{v})}^6}{g_{\frac{1}{N}\mathbf{u}}^{12}
g_{\frac{1}{N}\mathbf{v}}^{12}}\quad\textrm{by Lemma \ref{ffgg}}\\
&=&g\left\{
\frac{J^2(J-1)^3}{3^9}
\right\}^{m_N}\quad\textrm{with}~g=
\prod_{(\mathbf{u},\,\mathbf{v})\in P_N}
\frac{g_{\frac{1}{N}(\mathbf{u}+\mathbf{v})}^6
g_{\frac{1}{N}(\mathbf{u}-\mathbf{v})}^6}{g_{\frac{1}{N}\mathbf{u}}^{12}
g_{\frac{1}{N}\mathbf{v}}^{12}}.
\end{eqnarray*}
Since $f$ and $J$ belong to $\mathcal{F}_1$,
so does $g$. Moreover, since $g$ has neither a zero nor a pole on $\mathbb{H}$,
it is a nonzero rational number by Lemma \ref{modularunit1}.
Therefore we obtain
\begin{equation*}
f=k\left\{J^2(J-1)^3\right\}^{m_N}\quad
\textrm{for some}~k\in\mathbb{Q}\setminus\{0\}.
\end{equation*}
\end{proof}

\begin{lemma}\label{allst}
Let $K$ be an imaginary quadratic field and
$N$ be an integer with $N\geq2$.
If every prime factor of $N$ is inert in  $K$, then the principal ideal
$(s\tau^{}_{K}+t)\mathcal{O}_K$ is relatively prime to $N\mathcal{O}_K$
for all $s,\,t\in\mathbb{Z}$ such that
$\gcd(N,\,s,\,t)=1$.
\end{lemma}
\begin{proof}
We see that
\begin{equation*}
\mathrm{N}_{K/\mathbb{Q}}(s\tau^{}_{K}+t)=
(s\tau^{}_{K}+t)(s\overline{\tau}^{}_K+t)=
\tau^{}_{K}\overline{\tau}^{}_Ks^2+(\tau^{}_{K}+\overline{\tau}^{}_K)st+t^2=
c_Ks^2-b_Kst+t^2.
\end{equation*}
Now, we claim that $\mathrm{N}_{K/\mathbb{Q}}(s\tau^{}_{K}+t)$
is relatively prime to $N$. Indeed, we have two cases\,:
$d_K\equiv0\Mod{4}$ or $d_K\equiv1\Mod{4}$.
\begin{enumerate}
\item[Case 1.] Consider the case where $d_K\equiv0\Mod{4}$, and then
$b_K=0$ and $c_K=-\frac{d_K}{4}$.
Let $p$ be a prime factor of $N$. Since $p$ is inert in $K$,
it must be odd and satisfy $\left(\frac{d_K}{p}\right)=-1$.
If
\begin{equation*}
\mathrm{N}_{K/\mathbb{Q}}(s\tau^{}_{K}+t)=-\frac{d_K}{4}s^2+t^2\equiv0\Mod{p},
\end{equation*}
then the fact $\left(\frac{d_K}{p}\right)=-1$ forces us to get
$s\equiv0\Mod{p}$ and so $t\equiv0\Mod{p}$. But
this contradicts the fact $\gcd(N,\,s,\,t)=1$.
Therefore $\mathrm{N}_{K/\mathbb{Q}}(s\tau^{}_{K}+t)$ is relatively prime to $p$,
and hence to $N$.
\item[Case 2.] Let $d_K\equiv1\Mod{4}$, and so $b_K=1$ and
$c_K=\frac{1-d_K}{4}$. And, let $p$ be a prime factor of $N$.
Since $p$ is inert in $K$, we derive
\begin{equation*}
\left\{\begin{array}{ll}
d_K\equiv5\Mod{8} & \textrm{if}~p=2,\\
\left(\frac{d_K}{p}\right)=-1 & \textrm{if}~p>2.
\end{array}\right.
\end{equation*}
Then we find that
\begin{equation*}
\mathrm{N}_{K/\mathbb{Q}}(s\tau^{}_{K}+t)
=\frac{1-d_K}{4}s^2-st+t^2\equiv
\left\{\begin{array}{ll}
s^2+st+t^2\Mod{p} & \textrm{if}~p=2,\\
4'\left\{(s-2t)^2-d_Ks^2\right\}\Mod{p} & \textrm{if}~p>2,
\end{array}\right.
\end{equation*}
where $4'$ is an integer such that $4\cdot 4'\equiv1\Mod{p}$.
When $p=2$, we see that $s^2+st+t^2\equiv1\Mod{p}$ because
$s$ and $t$ are not both even.
When $p>2$, if $\mathrm{N}_{K/\mathbb{Q}}(s\tau^{}_{K}+t)\equiv0\Mod{p}$, then
the fact $\left(\frac{d_K}{p}\right)=-1$ yields that
$s\equiv0\Mod{p}$ and so $t\equiv0\Mod{p}$. But
this again contradicts $\gcd(N,\,s,\,t)=1$.
Hence $\mathrm{N}_{K/\mathbb{Q}}(s\tau^{}_{K}+t)$ is relatively prime to $p$,
and so to $N$.
\end{enumerate}
Therefore, the principal ideal
$(s\tau^{}_{K}+t)\mathcal{O}_K$ is relatively prime to $N\mathcal{O}_K$.
\end{proof}

\begin{theorem}\label{inertcase}
Let $K$ be an imaginary quadratic field other than $\mathbb{Q}(\sqrt{-1})$
and $\mathbb{Q}(\sqrt{-3})$, and let $N$ be an integer such that $N\geq2$.
Let $\omega$ be an element of $K$
so that $\omega+\mathcal{O}_K$
is a generator of the $\mathcal{O}_K$-module
$N^{-1}\mathcal{O}_K/\mathcal{O}_K$.
If every prime factor of $N$ is inert in $K$, then
we attain
\begin{equation*}
K_{(N)}=K\left({x^{}_{K,\,n}}(\omega)\right)\quad
\textrm{for every}~n\in\mathbb{Z}_{\geq0}.
\end{equation*}
\end{theorem}
\begin{proof}
Since $K_{(N)}$ is an abelian extension of $K$,
$K({x^{}_{K,\,n}}(\omega))$ is also an
abelian extension of $K$
containing ${x^{}_{K,\,n}}(\frac{1}{N})$
by Proposition \ref{CM} (ii).
And, since
\begin{equation*}
(s\tau^{}_{K}+t)\mathcal{O}_K\in P_K(N\mathcal{O}_K)\quad
\textrm{for all}~\begin{bmatrix}s&t\end{bmatrix}\in S_N
\end{equation*}
by Lemma \ref{allst}, we get by Proposition \ref{CM} (ii) that
\begin{equation*}
{x^{}_{K,\,n}}\left(\frac{s\tau^{}_{K}+t}{N}\right)\in
K({x^{}_{K,\,n}}(\omega))
\quad
\textrm{for all}~\begin{bmatrix}s&t\end{bmatrix}\in S_N.
\end{equation*}
We then deduce that
\begin{eqnarray*}
K({x^{}_{K,\,n}}(\omega))
&\ni&\prod_{(\mathbf{u},\,\mathbf{v})\in P_N}
\left({x^{}_{K,\,n}}\left(\frac{u_1\tau^{}_{K}+u_2}{N}\right)-
{x^{}_{K,\,n}}\left(\frac{v_1\tau^{}_{K}+v_2}{N}\right)\right)^6\\
&&\hspace{5cm}\textrm{with}~\mathbf{u}=\begin{bmatrix}u_1&u_2\end{bmatrix}~
\textrm{and}~\mathbf{v}=\begin{bmatrix}v_1&v_2\end{bmatrix}\\
&=&\prod_{(\mathbf{u},\,\mathbf{v})\in P_N}
\left(C_K^n\mathfrak{h}^{}_{K}\left(\varphi^{}_{K}\left(\left[\frac{u_1\tau^{}_{K}+u_2}{N}\right]\right)
\right)-
C_K^n\mathfrak{h}^{}_{K}\left(\varphi^{}_{K}\left(\left[\frac{v_1\tau^{}_{K}+v_2}{N}\right]\right)
\right)\right)^6\quad\textrm{by (\ref{xchp})}\\
&=&\prod_{(\mathbf{u},\,\mathbf{v})\in P_N}
\left\{
\left(\frac{C_K^n}{2^73^5}\right)^6\left(f_{\frac{1}{N}\mathbf{u}}(\tau^{}_{K})-
f_{\frac{1}{N}\mathbf{v}}(\tau^{}_{K})\right)^6\right\}\\
&&\hspace{5cm}\textrm{by the definitions of $\varphi^{}_{K},\,
\mathfrak{h}^{}_{K}$ and Fricke functions}\\
&=&
\left(
\frac{C_K^n}{2^73^5}\right)^{6m_N}
\left\{\prod_{(\mathbf{u},\,\mathbf{v})\in P_N}
\left(f_{\frac{1}{N}\mathbf{u}}-f_{\frac{1}{N}\mathbf{v}}\right)^6\right\}(\tau^{}_{K})\\
&=&
\left(
\frac{C_K^n}{2^73^5}\right)^{6m_N}
\left[k\left\{J^2(J-1)^3\right\}^{m_N}\right](\tau^{}_{K})
\quad\textrm{for some $k\in\mathbb{Q}\setminus\{0\}$ by Lemma \ref{normconstant}}\\
&=&k\left(\frac{C_K^{6n+1}}{2^{42}3^{30}}\right)^{m_N}.
\end{eqnarray*}
Therefore we achieve that
\begin{eqnarray*}
K({x^{}_{K,\,n}}(\omega))&=&
K\left(
{x^{}_{K,\,n}}(\omega),\,
k\left(\frac{C_K^{6n+1}}{2^{42}3^{30}}\right)^{m_N}
\right)\\
&=&H_K\left(
{x^{}_{K,\,n}}(\omega)
\right)\quad\textrm{by Lemma \ref{HKc}}\\
&=&K_{(N)}\quad\textrm{by (\ref{xchp}) and Proposition \ref{CM}}.
\end{eqnarray*}
\end{proof}

\section*{Acknowledgment}
Ho Yun Jung was supported by the research fund of Dankook University in 2021 and by the National Research Foundation of Korea(NRF) grant funded by the Korea government(MSIT)(No. 2020R1F1A1A01073055). Dong Hwa Shin was supported by Hankuk University of Foreign Studies Research Fund of 2021 and by the National Research Foundation of Korea(NRF) grant funded by the Korea government(MSIT)(No. 2020R1F1A1A01048633).\\

\section*{Funding}
Not applicable.

\section*{Availability of data and materials}
Not applicable.

\section*{Competing interests}
The authors declare that they have no competing interests.

\section*{Author's contributions}
All the authors contributes equally in this paper. All authors read and approved the final manuscript.

\bibliographystyle{amsplain}

\end{document}